\documentclass[12pt]{amsart}
\usepackage{amsaddr}
\usepackage[margin=3cm]{geometry}
\usepackage{amsmath,amssymb,amscd}
\usepackage{amsthm}
\usepackage{mathtools} 
\usepackage{ae} 
\usepackage{enumitem}
\usepackage{xcolor, graphicx}
\usepackage{tikz, ifthen}
\usepackage{pgfplots}
\pgfplotsset{compat=1.14}
\usetikzlibrary{shapes,arrows,backgrounds,fit,positioning,calc}

\theoremstyle{plain}
\newtheorem{theorem}{Theorem}
\newtheorem*{thma}{Theorem A}
\newtheorem*{thmb}{Theorem B}
\newtheorem*{thmc}{Theorem C}

\newtheorem{lemma}[theorem]{Lemma}
\newtheorem{proposition}[theorem]{Proposition}
\newtheorem{corollary}[theorem]{Corollary}
\theoremstyle{definition}

\newtheorem{remark}[theorem]{Remark}

\DeclareMathOperator{\Cells}{Cells}
\DeclareMathOperator{\HausDist}{HausDist}
\DeclareMathOperator{\dist}{dist}
\DeclareMathOperator{\diam}{diam}

\title{Dynamics on dendrites with closed endpoint sets}
\author{Samuel Roth}
\subjclass[2010]{Primary: 37B45, Secondary: 54C25, 37B05}
\keywords{dendrite, isometric embedding, Li-Yorke chaos, distributional chaos}

\begin{document}
\begin{abstract}
We construct dendrites with endpoint sets isometric to any totally disconnected compact metric space. This allows us to embed zero-dimensional dynamical systems into dendrites and solve a problem regarding Li-Yorke and distributional chaos.
\end{abstract}
\maketitle

\section{Introduction}

A \emph{dendrite} is a locally connected continuum (compact connected metric space) with no subspace homeomorphic to the circle. The \emph{endpoints} and \emph{branchpoints} of a dendrite $X$ are defined, respectively, by
\begin{equation}\label{EB}
\begin{aligned}
E(X) &= \left\{ x ~|~ X\setminus\{x\} \text{ is connected}\right\},\\
B(X) &= \left\{ x ~|~ X\setminus\{x\} \text{ has at least 3 connected components}\right\}.
\end{aligned}
\end{equation}

Dendrites in which the endpoint set $E(X)$ is a closed set were investigated in~\cite{Closed}. They were characterized in terms of two topological properties: branchpoints can accumulate only at endpoints; and for each branchpoint $x$, $X\setminus\{x\}$ must have only finitely many connected components. Equivalently, $X$ must not contain a homeomorphic copy of either of the two dendrites shown in Figure~\ref{fig:forbidden}.

\begin{figure}[b!!]
\begin{minipage}[b]{.45\textwidth}\centering
\begin{tikzpicture}
\foreach \i in {1,...,8} {\draw [thick] (0,0) -- (17*\i-17:{15/(\i+2)});}
\foreach \i in {9,...,11} {\draw [fill] (17*\i-17:{15/(\i+2)}) circle [radius=0.02];}
\end{tikzpicture}
\end{minipage}%
\begin{minipage}[b]{.45\textwidth}\centering
\begin{tikzpicture}
\draw (-3,0) -- (3,0);
\foreach \i in {1,...,50} {\draw [thick] (3/\i,3/\i) -- (3/\i,0);}
\end{tikzpicture}
\end{minipage}
\caption{The two subdendrites which cannot appear in a dendrite with a closed set of endpoints~\cite{Closed}. See Section~\ref{sec:extension} for precise definitions.}
\label{fig:forbidden}
\end{figure}
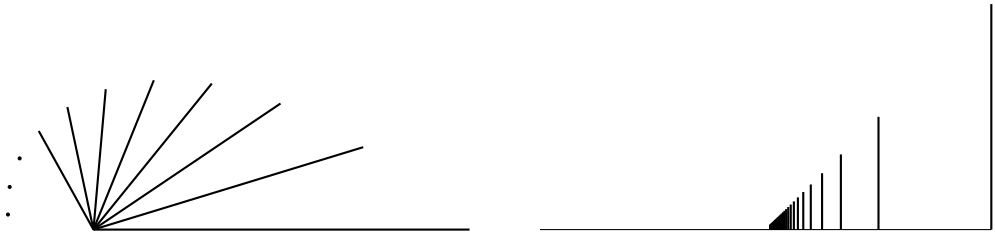

In this paper we start with a related problem: to characterize (up to isometry) the compact metric spaces which can occur as the endpoint set of a dendrite. Note that in this paper a dendrite $X$ always comes equipped with a metric, so that the set of endpoints is a metric subspace. It is compact if and only if $E(X)$ is closed in $X$. Thus, we are interested precisely in the class of dendrites studied in~\cite{Closed}.

It is easy to show that the endpoint set of a dendrite is \emph{totally disconnected}, i.e. the only nonempty connected subsets are the singletons. Our main construction (section~\ref{sec:construction}) shows that this condition is also sufficient for a compact metric space to arise as the endpoint set of a dendrite:

\begin{thma}\label{th:endpoints}
Every compact, totally disconnected metric space is isometric with the (necessarily closed) endpoint set of some dendrite.
\end{thma}

In the next part of the paper, Section~\ref{sec:extension}, we ask about extending dynamical systems from the closed endpoint set of a dendrite to the whole dendrite. There is a simple topological result that any self-map on a closed subset of a dendrite can be extended to the whole dendrite~\cite[Lemma 2.4]{BHS}, but we wish to make the extension in a more controlled way. A point $x\in X$ is \emph{eventually fixed} for the map $f:X\to X$ if for some $n\geq0$, the iterated image $f^n(x)$ is a fixed point, i.e. $f^n(x)=f^{n+1}(x)$. Using techniques from~\cite{Spitalsky} and the results from~\cite{Closed} we prove

\begin{thmb}\label{th:extension}
Let $Y$ be a dendrite with a closed set of endpoints. Any continuous map $f:E(Y)\to E(Y)$ can be extended to a continuous map $F:Y \to Y$ in such a way that each point of $Y\setminus E(Y)$ is eventually fixed.
\end{thmb}

Combining Theorems A and B gives us a way to embed dynamical systems into dendrites while preserving many basic dynamical properties (topological entropy, presence/absence of Li-Yorke chaos, etc.). This extends the work of Ko\v{c}an, Kurkov\'{a}, and M\'{a}lek, who use the Gehman dendrite and its subdendrites to represent every shift-invariant subset of $\{0,1\}^\mathbb{N}$ -- in this way, any subshift can be embedded into a dendrite map so that all other points are eventually fixed \cite{KKM}. Shift spaces are a rich source of examples and counterexamples in dynamical systems, and this construction has led to an abundance of dendrite maps with various combinations of chaotic and non-chaotic behavior \cite{Drwiega, Drwiega19, Koc, KKM14, Zuzka}. Our construction allows us to start with other zero-dimensional systems and gives us isometric embeddings rather than just topological ones. As an application, in Section~\ref{sec:solution} we solve an open problem~\cite[Question 2]{Zuzka} regarding distributional chaos and Li-Yorke chaos on dendrites (we defer the definitions to Section~\ref{sec:solution}):

\begin{thmc}\label{th:DC3}
There is a dynamical system on a dendrite which is DC3 but not Li-Yorke chaotic.
\end{thmc}

\section{Definitions and notation}

Let $(X,d)$ be a compact metric space.  We denote the open ball of radius $r>0$ around a point $x$ as $B_X(x,r)$. A \emph{neighborhood} of $x$ is any set (not necessarily open) which contains a whole open ball around $x$. Any set which is both open and closed in $X$ will be called \emph{clopen}. A point $x$ is \emph{isolated} in $X$ if the singleton $\{x\}$ is itself a neighborhood of $x$. The \emph{diameter} of a set $A\subset X$ is $$\diam(A)=\sup_{a, a' \in A} d(a, a').$$ The (non-symmetric) \emph{distance} of a set $A$ from a set $B$ is $$\dist(A,B)=\adjustlimits\sup_{a\in A} \inf_{b\in B} d(a,b).$$ The \emph{Hausdorff distance} between two sets is $$\HausDist(A,B) = \max\{ \dist(A,B), \dist(B,A) \}.$$ The collection of all nonempty closed subsets of $X$ equipped with the Hausdorff distance is itself a compact metric space~\cite[Theorem 4.13]{Nadler}. It is sometimes denoted $2^X$ and referred to as the \emph{hyperspace} of $X$.

An \emph{arc} in a topological space $X$ is a subspace of $X$ homeomorphic with the unit interval $[0,1]$. If an arc minus its endpoints is an open set in $X$ then we refer to it as a \emph{free arc}.

In~\eqref{EB} we gave a definition of endpoints which applies only to dendrites. More generally, in a continuum $X$ having more than one point we call $x$ an \emph{endpoint} if there are arbitrarily small neighborhoods $U$ of $x$ such that the boundary of $U$ is a singleton. This agrees with what we defined before by~\cite[Theorem 10.13]{Nadler}. Moreover, we define a \emph{cutpoint} in a continuum $X$ as any point $x$ such that $X\setminus\{x\}$ is disconnected. A continuum is a dendrite if and only if each of its points is an endpoint or a cutpoint~\cite[Theorem 10.7]{Nadler}.

As mentioned before, a topological space $X$ is called \emph{totally disconnected} if its only nonempty connected subsets are singletons. A metric space is called \emph{zero-dimensional} if each of its points has clopen neighborhoods with arbitrarily small diameters. A compact metric space is totally disconnected if and only if it is zero-dimensional, see, eg.~\cite{Hurewicz}.

\section{Dendrites with a prescribed endpoint set}\label{sec:construction}

Given a totally disconnected compact metric space $(X,d)$, the goal of this section is to construct a dendrite $Y$ with a metric $\rho$ so that the endpoint set $E(Y)$ is isometric with $X$.

The construction takes place in several stages. Subsection~\ref{sec:chains} uses chain-connected sets to decompose $X$ into ``cells.'' Subsection~\ref{sec:skeleton} arranges these cells in a graph which forms the skeleton of the dendrite. Subsection~\ref{sec:Yrho} defines $Y$ by replacing the edges of the graph with arcs. It also defines the metric $\rho$. Subsection~\ref{sec:proof} proves that $Y$ truly is a dendrite with endpoint set isometric to $X$.

\subsection{Chains and cells}\label{sec:chains}\strut\\
Let $(X,d)$ be any compact metric space. A \emph{$\theta$-chain} is a finite sequence of points $x_1, x_2, \ldots, x_n \in X$ such that $d(x_i,x_{i+1})\leq\theta$ for all $i<n$. We say that two points $x, y\in X$ are \emph{connected by a $\theta$-chain}, and write $x \sim_\theta y$, if there is a $\theta$-chain $x_1, \ldots, x_n$ with $x_1=x$ and $x_n=y$. Observe that $\sim_\theta$ is an equivalence relation on $X$. The elements of the induced partition are called \emph{$\theta$-cells}. The collection of all \emph{cells} of $X$ is
\begin{equation*}
\Cells(X) = \left\{ A \subseteq X ~|~ \text{$A$ is a $\theta$-cell for some $\theta\geq0$} \right\}.
\end{equation*}

Two examples of compact metric spaces are shown in Figure~\ref{fig:decomposition}. The figure shows several cells in each space as gray shaded regions.

\begin{figure}[htb!!]
\begin{minipage}[c]{.45\textwidth}\centering
\scalebox{0.65}{
\begin{tikzpicture}[
declare function={
fx(\x,\y,\i) = (\i==1)*(2*cos(150)+1/3.5*(\x)) + (\i==2)*(2*cos(270)+1/3.5*(\x)) + (\i==3)*(2*cos(30)+1/3.5*(\x));
fy(\x,\y,\i) = (\i==1)*(2*sin(150)+1/3.5*(\y)) + (\i==2)*(2*sin(270)+1/3.5*(\y)) + (\i==3)*(2*sin(30)+1/3.5*(\y));
}
]
\path [fill=gray!10] (0,0) circle [radius=3.7];
\foreach \i in {1,...,3} {
 \path [fill=gray!25] ({fx(0,0,\i)},{fy(0,0,\i)}) circle [radius=1.4];
}
\foreach \i in {1,...,3} {
 \foreach \j in {1,...,3} {
  \path [fill=gray!50] ({fx(fx(0,0,\i),fy(0,0,\i),\j)},{fy(fx(0,0,\i),fy(0,0,\i),\j)}) circle [radius=0.45];
}}
\foreach \i in {1,...,3} {
 \foreach \j in {1,...,3} {
  \foreach \k in {1,...,3} {
   \draw [fill] ({fx(fx(fx(0,0,\i),fy(0,0,\i),\j),fy(fx(0,0,\i),fy(0,0,\i),\j),\k)},{fy(fy(fx(0,0,\i),fy(0,0,\i),\j),fy(fx(0,0,\i),fy(0,0,\i),\j),\k)}) circle [radius=0.1];
}}}
\end{tikzpicture}
}
\end{minipage}%
\begin{minipage}[c]{.45\textwidth}\centering
\scalebox{0.8}{\begin{tikzpicture}
\foreach \i in {1,...,20} {\draw [fill] (5/\i,0) circle [radius=0.1] node (\i) {};}
\draw [fill] (0,0) circle [radius=0.2] node (0) {};
\begin{pgfonlayer}{background}
  \node[fit=(0)(1), rounded corners=0.3cm, inner xsep=50pt, minimum height=2cm, fill=gray!10] {};
  \node[fit=(0)(2), rounded corners=0.25cm, inner xsep=40pt, minimum height=1.3cm, fill=gray!25] {};
  \node[fit=(1), rounded corners=0.25cm, inner xsep=10pt, minimum height=1.3cm, fill=gray!25] {};
  \node[fit=(0)(3), rounded corners=0.2cm, inner xsep=7pt, minimum height=0.7cm, fill=gray!50] {};
  \node[fit=(2), rounded corners=0.2cm, inner xsep=5pt, minimum height=0.7cm, fill=gray!50] {};
\end{pgfonlayer}
\end{tikzpicture}}
\end{minipage}
\caption{Two examples of the division of a space into cells. The space on the left is a Cantor set in the plane. The space on the right is the set $\{\frac1n ~|~ n\in\mathbb{N}\} \cup \{0\}$ in the real line.}
\label{fig:decomposition}
\end{figure}
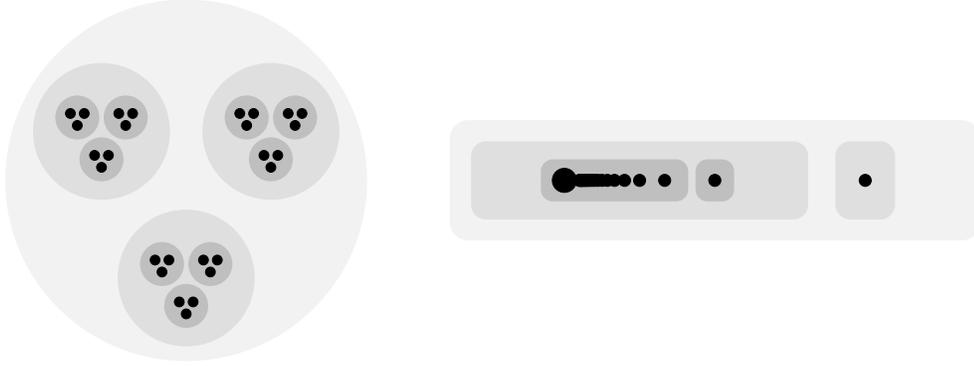

We describe the basic properties of cells through a sequence of lemmas.

\begin{lemma}\label{lem:cells}
In any compact metric space $(X,d)$,
\begin{enumerate}[label=(\roman*)]
\item\label{it:0-cells} All $0$-cells are singletons,
\item\label{it:clopen} All $\theta$-cells, $\theta>0$ are both closed and open in $X$,
\item\label{it:finite} For fixed $\theta>0$ there are only finitely many $\theta$-cells,
\item\label{it:contained} Two cells are either disjoint, or else one is contained in the other,
\item\label{it:union} If $0<\eta\leq\theta$, then each $\theta$-cell is a finite union of $\eta$-cells, and
\item\label{it:greater} If a $\theta$-cell strictly contains an $\eta$-cell, then $\theta > \eta$.
\end{enumerate}
\end{lemma}

\noindent Lemma~\ref{lem:cells} follows immediately from the definitions and the compactness of $X$. The proof is omitted.

%
%
%
%
%
%

\begin{lemma}\label{lem:diam}
If the compact metric space $(X,d)$ is totally disconnected, then for each $r>0$ there exists $\theta>0$ such that each $\theta$-cell has diameter less than $r$.
\end{lemma}
\begin{proof}
Suppose to the contrary that there exists $r>0$ and for each $n\in\mathbb{N}$ a $\frac1n$-cell $A_n$ with diameter greater than or equal to $r$. Choose points $x_n, y_n \in A_n$ with $d(x_n, y_n)\geq r$. Using the compactness of $X$ and passing to subsequences if necessary we may assume that $x_n \to x$ and $y_n \to y$ in $X$ with $d(x,y)\geq r$. Since $X$ is totally disconnected we can write it as a union of two clopen sets $X=U\cup V$ with $x\in U$, $y\in V$, see~\cite[Lemma 7.11]{Nadler}. Let $\theta = \min \{ d(u,v) ~|~ u\in U, v\in V\}.$ By compactness of $U\times V$, we get $\theta>0$. Whenever $n$ is large enough that $\frac1n < \theta$, we see that either $A_n\cap U=\emptyset$ or $A_n\cap V=\emptyset$. This contradicts the fact that the points $x_n, y_n \in A_n$ are converging to the points $x\in U$, $y\in V$.
\end{proof}

Now we equip $\Cells(X)$ with the Hausdorff metric and consider its properties as a topological space.

\begin{lemma}\label{lem:compact}
If the compact metric space $(X,d)$ is totally disconnected, then $\Cells(X)$ is compact and each non-singleton cell is isolated in $\Cells(X)$.
\end{lemma}
\begin{proof}
To show compactness it suffices to show that $\Cells(X)$ is closed in $2^X$. Let $(C_n)$ be a sequence of cells converging in the Hausdorff metric to a set $K\in 2^X$. If some cell $C$ appears in the sequence $(C_n)$ infinitely often, then $K=C\in\Cells(X)$. If infinitely many members of $(C_n)$ are distinct, then by Lemma~\ref{lem:cells}~\ref{it:finite}, \ref{it:union}, and Lemma~\ref{lem:diam}, $\diam(C_n) \to 0$ and $K$ is a singleton. Then also $K\in\Cells(X)$. This also shows that non-singleton cells are isolated.
\end{proof}

\begin{remark}
Lemma~\ref{lem:compact} shows that $\Cells(X)$ has similar properties as $E(Y)\cup B(Y)$ when $Y$ is a dendrite with a closed endpoint set, see Proposition~\ref{prop:char}. The singleton cells correspond with endpoints and the non-singleton cells correspond with branchpoints.
\end{remark}

\subsection{The ``skeleton'' of the dendrite}\label{sec:skeleton}\strut\\
The cells of a compact metric space $(X,d)$ are partially ordered by inclusion. If $A\subset B$ is a proper subcell, then we call $A$ a \emph{descendent} of $B$ and $B$ an \emph{ancestor} of $A$. If $A\subset B$ is a maximal proper subcell, then we call $A$ a \emph{child} of $B$ and $B$ the \emph{parent} of $A$.

We can get a ``skeleton'' for the dendrite we wish to construct by forming a graph with $\Cells(X)$ as its vertex set and with edges joining each cell to its children. Two possibilities for this graph are shown in Figure~\ref{fig:graph}.

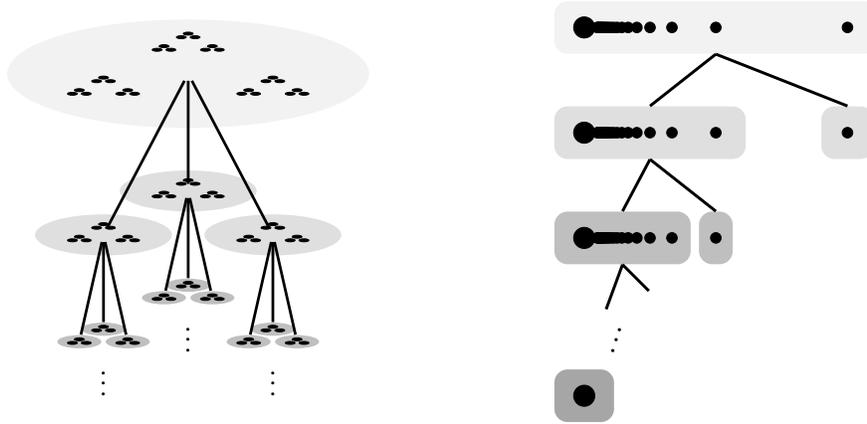
\begin{figure}[htb!!]
\begin{minipage}[c]{.45\textwidth}\centering
\scalebox{0.65}{
\begin{tikzpicture}[
declare function={
fx(\x,\y,\i) = (\i==1)*(2*cos(150)+1/3.5*(\x)) + (\i==2)*(2*cos(270)+1/3.5*(\x)) + (\i==3)*(2*cos(30)+1/3.5*(\x));
fy(\x,\y,\i) = (\i==1)*(2*sin(150)+1/3.5*(\y)) + (\i==2)*(2*sin(270)+1/3.5*(\y)) + (\i==3)*(2*sin(30)+1/3.5*(\y));
}
]
\path [fill=gray!10, yscale=-0.3] (0,-20) circle [radius=3.7] node (root) {};
\foreach \i in {1,...,3} {
 \foreach \j in {1,...,3} {
  \foreach \k in {1,...,3} {
   \draw [fill, yscale=-0.3] ({fx(fx(fx(0,0,\i),fy(0,0,\i),\j),fy(fx(0,0,\i),fy(0,0,\i),\j),\k)},{-20+fy(fy(fx(0,0,\i),fy(0,0,\i),\j),fy(fx(0,0,\i),fy(0,0,\i),\j),\k)}) circle [radius=0.1];
}}}
\foreach \i in {1,...,3} {
 \path [fill=gray!25, yscale=-0.3] ({fx(0,0,\i)},{-10+fy(0,0,\i)}) circle [radius=1.4] node (\i) {};
}
\foreach \i in {1,...,3} {
 \foreach \j in {1,...,3} {
  \foreach \k in {1,...,3} {
   \draw [fill, yscale=-0.3] ({fx(fx(fx(0,0,\i),fy(0,0,\i),\j),fy(fx(0,0,\i),fy(0,0,\i),\j),\k)},{-10+fy(fy(fx(0,0,\i),fy(0,0,\i),\j),fy(fx(0,0,\i),fy(0,0,\i),\j),\k)}) circle [radius=0.1];
}}}
\foreach \i in {1,...,3} {
 \foreach \j in {1,...,3} {
  \path [fill=gray!50, yscale=-0.3] ({fx(fx(0,0,\i),fy(0,0,\i),\j)},{-3+fy(fx(0,0,\i),fy(0,0,\i),\j)}) circle [radius=0.45] node (\i\j) {};
}}
\foreach \i in {1,...,3} {
 \foreach \j in {1,...,3} {
  \foreach \k in {1,...,3} {
   \draw [fill, yscale=-0.3] ({fx(fx(fx(0,0,\i),fy(0,0,\i),\j),fy(fx(0,0,\i),fy(0,0,\i),\j),\k)},{-3+fy(fy(fx(0,0,\i),fy(0,0,\i),\j),fy(fx(0,0,\i),fy(0,0,\i),\j),\k)}) circle [radius=0.1];
}}}
\foreach \i in {1,...,3} {\draw [ultra thick] (root) -- (\i);}
\foreach \i in {1,...,3} { \foreach \j in {1,...,3} {\draw [ultra thick] (\i) -- (\j\i);}}
\foreach \i in {1,...,3} {\path (\i) -- +(0,-3) node () {\rotatebox{90}{\textbf{\ldots}}};}
\end{tikzpicture}
}
\end{minipage}%
\begin{minipage}[c]{.45\textwidth}\centering
\scalebox{0.7}{
\begin{tikzpicture}
\foreach \i in {1,...,20} {\draw [fill] (5/\i,0) circle [radius=0.1] node (\i) {};}
\draw [fill] (0,0) circle [radius=0.2] node (0) {};
\foreach \i in {1,...,20} {\draw [fill] (5/\i,-2) circle [radius=0.1] node (\i-1) {};}
\draw [fill] (0,-2) circle [radius=0.2] node (0-1) {};
\foreach \i in {2,...,20} {\draw [fill] (5/\i,-4) circle [radius=0.1] node (\i-2) {};}
\draw [fill] (0,-4) circle [radius=0.2] node (0-2) {};
\draw [fill] (0, -7) circle [radius=0.2] node (0-4) {};
\begin{pgfonlayer}{background}
  \node[fit=(0)(1), rounded corners=0.25cm, inner xsep=12pt, minimum height=1cm, fill=gray!10] (A) {};
  \node[fit=(0-1)(2-1), rounded corners=0.25cm, inner xsep=12pt, minimum height=1cm, fill=gray!25] (B) {};
  \node[fit=(1-1), rounded corners=0.25cm, inner xsep=10pt, minimum height=1cm, fill=gray!25] (C) {};
  \node[fit=(0-2)(3-2), rounded corners=0.25cm, inner xsep=9pt, xshift=-3pt, minimum height=1cm, fill=gray!50] (D) {};
  \node[fit=(2-2), rounded corners=0.25cm, inner xsep=5pt, minimum height=1cm, fill=gray!50] (E) {};
  \node[fit=(0-4), rounded corners=0.25cm, inner xsep=12pt, minimum height=1cm, fill=gray!70] (G) {};
\end{pgfonlayer}
\draw [ultra thick] (A.south) -- (B.north);
\draw [ultra thick] (A.south) -- (C.north);
\draw [ultra thick] (B.south) -- (D.north);
\draw [ultra thick] (B.south) -- (E.north);
\node at ($(D.south)!.5!(G.north)$) (F) {};
\node at ($(D.south)!.7!(G.north east)$) () {\rotatebox{72}{\textbf{\ldots}}};
\draw [ultra thick] (D.south) -- (F);
\draw [ultra thick] (D.south) -- ++(.5,-.5);
\end{tikzpicture}
}
\end{minipage}
\caption{Parent/child relations among the cells of $X$ (cf. Figure~\ref{fig:decomposition}).}
\label{fig:graph}
\end{figure}

\begin{lemma}\label{lem:graph}
Let $(X,d)$ be a totally disconnected compact metric space and let $A\in\Cells(X)$.
\begin{enumerate}[label=(\roman*)]
\item\label{it:nonsingleton} If $A$ is not a singleton, then $A$ has finitely many children. They are clopen sets in $X$ and their union is $A$.
\item\label{it:clopencell} If $A$ is clopen in $X$ and $A\neq X$, then $A$ has exactly one parent cell and finitely many ancestors.
\item\label{it:limitcell} If $A=\{x\}$ with $x$ not isolated in $X$, then $A$ has no parent cell and its ancestors can be arranged in a decreasing sequence $X=C_1 \supset C_2 \supset \cdots$ with $C_n \to A$ in the Hausdorff metric as $n\to\infty$.
\end{enumerate}
\end{lemma}

\begin{proof}
\emph{\ref{it:nonsingleton}:}
We know that $A$ is a $\theta$-cell for some $\theta>0$. Since $X$ is totally disconnected, we can write $A=U \cup V$ as the disjoint union of two closed, nonempty subsets. Since $U,V$ are closed in $A$ and $A$ is closed in $X$, we get that $U, V$ are compact. Then
$\eta := \inf \{ d(u,v) ~|~ u\in U, v\in V \}$
is positive. Fix $u\in U$, $v\in V$. Then $u, v$ are connected by a $\theta$-chain in $A$, but cannot be connected by an $\eta$-chain in $A$. Therefore $\eta<\theta$. Since $A$ is a $\theta$-cell, $d(a,x)>\theta>\eta$ for all $a\in A$, $x\in X\setminus A$. Therefore $u, v$ cannot be connected by an $\eta$-chain even within the larger space $X$, so $u\not\sim_\eta v$ in $X$. It follows that $A$ is not an $\eta$-cell. By Lemma~\ref{lem:cells}~\ref{it:union}, $A$ is a finite union of $\eta$-cells. Using those $\eta$-cells we can take unions, and the maximal unions which are in $\Cells(X)$ but are not equal to $A$ are precisely the children of $A$.

\emph{\ref{it:clopencell}:}
Even if $A$ is a singleton, the clopen condition implies that $A$ is a $\theta$-cell for some $\theta>0$. The ancestors of $A$ are all finite unions of $\theta$-cells by Lemma~\ref{lem:cells}~\ref{it:union}, so by Lemma~\ref{lem:cells}~\ref{it:finite} there are only finitely many of them. By Lemma~\ref{lem:cells}~\ref{it:contained}, the ancestors of $A$ are totally ordered by inclusion. The smallest ancestor cell is then the (unique) parent of $A$.

\emph{\ref{it:limitcell}:}
Let $C_1=X$. Once a non-singleton ancestor $C_n$ has been defined, let $C_{n+1}$ be the child of $C_n$ containing $x$. Then $C_{n+1}$ is clopen in $X$, so it is not equal to $\{x\}$. This allows us to continue the induction. In this way we get a decreasing sequence $C_1 \supset C_2 \supset \cdots$ of ancestors of $\{x\}$. By Lemma~\ref{lem:compact}, $C_n \to \{x\}$ in the space $\Cells(X)$ under the Hausdorff metric.
\end{proof}

\subsection{Construction of the dendrite $(Y,\rho)$}\label{sec:Yrho}\strut\\
Again, let $(X,d)$ be a totally disconnected compact metric space. We create the space $Y$ from the ``skeleton'' graph described above be replacing each edge with an arc. At the set-theoretic level we can accomplish this by taking
\begin{equation*}
Y = \Cells(X) \cup \{ (A, B, t) ~|~ A, B\in \Cells(X), t\in[0,1], B \text{ is a child of }A \}
\end{equation*}
and identifying $(A,B,0)$ with $A$ and $(A,B,1)$ with $B$. The set of points $(A,B,t)$, $0\leq t\leq 1$ will be denoted $[A, B]$ (interval notation) and will be called an \emph{edge} of $Y$.

Now we introduce a metric $\rho$ on $Y$. For $A, B\in\Cells(X)$ we define
\begin{equation}\label{rho_cell_cell}
\rho(A,B) = \HausDist(A,B)
\end{equation}
where the Hausdorff distance is measured between $A$ and $B$ as closed subsets of $X$. On the edge $[A,B]$ if $y=(A,B,t)$ and $y'=(A,B,t')$ then we put
\begin{equation}\label{rho_edge}
\rho(y,y')=|t-t'| / \HausDist(A,B),
\end{equation}
so that each edge is an isometric copy of a Euclidean interval of the right length.  If $a\in[A_1,A_2]$ and $B\in\Cells(X)$ then we put
\begin{equation}\label{rho_point_cell}
\rho(a,B) = \min_{i\in\{1,2\}} \rho(a,A_i) + \rho(A_i,B).
\end{equation}
Finally, if $a, b$ belong to distinct edges $[A_1, A_2]$, $[B_1, B_2]$, respectively, then we put
\begin{equation}\label{rho_point_point}
\rho(a,b) = \min_{i,j\in\{1,2\}} \rho(a,A_i) + \rho(A_i,B_j) + \rho(B_j, b).
\end{equation}
In other words, we look for the shortest path moving first along the edges to their endpoints, and then jumping between endpoints.

\begin{lemma}\label{lem:rho}
$\rho$ is a metric on the set $Y$.
\end{lemma}
\begin{proof}
It is easy to see that $\rho$ is symmetric and positive definite. To verify the triangle inequality, suppose that $a\in[A_1,A_2]$, $b\in[B_1,B_2]$, $c\in[C_1,C_2]$ are points from three distinct edges. Choosing the indices $i,j,k,l\in\{1,2\}$ to minimize the first sum we get
\begin{equation}\label{triangle}
\begin{aligned}
\rho(a,b)+\rho(b,c)
&= \rho(a,A_i) + \rho(A_i,B_j) + \underbrace{\rho(B_j,b) + \rho(b,B_k)}_{\mathclap{\text{Use triangle ineq. in }[B_1,B_2]}} + \rho(B_k,C_l) + \rho(C_l,c) \\
&\geq \rho(a,A_i) + \underbrace{\rho(A_i,B_j) + \rho(B_j,B_k) + \rho(B_k,C_l)}_{\text{Use triangle ineq. for Hausdorff distance}} + \rho(C_l,c) \\
&\geq \rho(a,A_i) + \rho(A_i,C_l) + \rho(C_l,c) \\
& \geq \rho(a,c).
\end{aligned}
\end{equation}
There are also cases when some of the points $a,b,c$ belong to the same edge, or when some of these points are cells of $X$. We leave them to the reader.
\end{proof}

\begin{lemma}\label{lem:freearc}
Each edge $[A,B]$ in $Y$ is a free arc in $Y$.
\end{lemma}
\begin{proof}
We must show that $[A,B]\setminus\{A,B\}$ is open in $Y$. If $a\in[A,B]\setminus\{A,B\}$, then the ball $B_Y(a,r)$ where $r=\min\{\rho(a,A),\rho(a,B)\}$ is contained in $[A,B]\setminus\{A,B\}$ by the definition of the metric $\rho$.
\end{proof}

\subsection{Verification that $Y$ is a dendrite}\label{sec:proof}\strut\\
We are now ready to state and prove the main theorem of this section.

\begin{thma}
Let $(X,d)$ be a totally disconnected compact metric space. Then there exists a dendrite $(Y, \rho)$ whose endpoint set $E(Y)$ is isometric to $X$.
\end{thma}
\begin{proof}
There are two trivial cases. If $X$ is empty or a singleton, then $Y=X$ is the desired dendrite. From now on we assume that $X$ contains at least two points.

Construct $(Y,\rho)$ from $(X,d)$ as described in Section~\ref{sec:Yrho}. Let $i:X\to Y$ be the map $i(x)=\{x\}$. It is an isometry of $X$ onto its image by~\eqref{rho_cell_cell}. We will show that $Y$ is a dendrite and $i(X)=E(Y)$.

\emph{Compactness of $Y$:} Let $\mathcal{U}$ be any open cover of $Y$. We know from Lemma~\ref{lem:compact} that $\Cells(X)$ is a compact subspace of $Y$, so it can be covered by a finite subset $\mathcal{U}_1 \subset \mathcal{U}$. By compactness there is $r>0$ such that $\mathcal{U}_1$ covers all of the balls $B_Y(A,r)$, $A\in\Cells(X)$. If $A\subset X$ is a cell with diameter less than $r$ and $B$ is a child of $A$, then by the definition of $\rho$ the whole edge $[A,B]$ is contained in the ball $B_Y(A,r)$. By Lemmas~\ref{lem:cells},~\ref{lem:diam}, and~\ref{lem:graph} there are only finitely many edges in $Y$ not already covered by $\mathcal{U}_1$. But a finite union of arcs in $Y$ is again compact, and so is covered by a finite subset $\mathcal{U}_2\subset\mathcal{U}$. Then $\mathcal{U}_1\cup\mathcal{U}_2$ covers all of $Y$.

\emph{Connectedness of $Y$:}
We show that each point of $Y$ can be joined by an arc in $Y$ to the point $X\in\Cells(X)$. First let $A\in\Cells(X)$. We use Lemma~\ref{lem:graph}. If $A$ is clopen in $X$, then its ancestors form a finite sequence $X=A_1 \supset A_2 \supset \cdots \supset A_n=A$. Then $\bigcup_{i<n} [A_i, A_{i+1}]$ is an arc in $Y$ joining $X$ to $A$. On the other hand, if $A=\{x\}$ with $x$ not isolated in $X$, then its ancestors form an infinite sequence $X=A_1 \supset A_2 \supset \cdots$ with $A_n \to A$. Then the closure of $\bigcup_{i\in\mathbb{N}} [A_i, A_{i+1}]$ is an arc in $Y$ joining $X$ to $A$. Finally, if we choose a point $y$ from an edge $[A,B]$, then the arc in $Y$ joining $X$ to $B$ passes through $y$, so there is a subarc joining $X$ to $y$ in $Y$.

\emph{Endpoints of $Y$:}
We have not yet shown that $Y$ is a dendrite. We only know that it is a continuum. We will show that each singleton cell $A=\{x\}\in Y$ is an endpoint in the sense that it has arbitrarily small neighborhoods with a 1-point boundary.

If $A=\{x\}$ with $x$ isolated in $X$, then $A$ is isolated in $\Cells(X)$ and clopen in $X$. Then the single edge in $Y$ connecting $A$ to its parent cell is a neighborhood of $A$. Within this edge we get arbitrarily small neighborhoods of $A$ with a 1-point boundary in $Y$.

Now suppose $A=\{x\}$ with $x$ not isolated in $X$. Let $\epsilon>0$ be given. By Lemma~\ref{lem:graph} there is an ancestor $C$ of $A$ with $\diam_X(C)<\epsilon/3$. Then $\Cells(C)$ is the subset of $\Cells(X)$ consisting of $C$ and all its descendents. Let $K=K(C)$ be the subset of $Y$ consisting of $\Cells(C)$ together with all edges whose endpoints are both in $\Cells(C)$. The Hausdorff distance between any two cells in $C$ is at most $\epsilon/3$, so by the definition of $\rho$ the diameter of $K$ is at most $\epsilon$. Next we show that $K$ is a neighborhood of $A$. For $B$ the parent cell of $C$ we have $B_Y(A,\rho(A,C))\cap[B,C]=\emptyset$. Let $\theta>0$ be such that $C$ is a $\theta$-cell. Then $B_X(x,\theta)\subset C$. Therefore $B_Y(A,\theta)\cap(\Cells(X)\setminus\Cells(C))=\emptyset$. Besides $[B,C]$, every edge in $Y$ not contained in $K$ has both endpoints in $\Cells(X)\setminus\Cells(C)$. Taken together, this shows that $B_Y(A,r)\subset K$, where $r=\min\{\theta,\rho(A,C)\}$. Next we calculate the boundary of $K$. Just as $K$ is a neighborhood of $A$, so also it is a neighborhood of each of the singleton cells it contains (they are all descendents of $C$). By Lemmas~\ref{lem:compact},~\ref{lem:graph}, and~\ref{lem:freearc}, $K$ is a neighborhood of each of its points except $C$. Thus $K$ has a 1-point boundary.

\emph{Cutpoints of $Y$:}
We would like to show that $Y\setminus\{y\}$ is disconnected for each $y\in Y$ which is not a singleton cell.

If $C\in\Cells(X)$ is not a singleton and $C\neq X$, then we can form $K(C)$ as above and the only point in its boundary is $C$. Then in $Y\setminus\{C\}$ the set $K(C)\setminus\{C\}$ is both open and closed and is not the whole space.

Next consider the cell $X$ itself. By Lemma~\ref{lem:graph} it has finitely many children $C_1, \ldots, C_n$. Form the set $K(C_1) \cup [X,C_1]$. Its boundary in $Y$ is the single point $X$. Then in $Y\setminus\{X\}$ the set $K(C_1)\cup[X,C_1]\setminus\{X\}$ is both open and closed and is not the whole space.

Finally, consider a point $y\in Y\setminus\Cells(X)$. It belongs to an edge $[A,B]$. If $B$ is a singleton cell then we take the subarc joining $y$ to $B$. If $B$ is not a singleton then we form $K(B)$ and join to it the subarc joining $y$ to $B$. In either case we have formed a set whose boundary in $Y$ is the single point $y$, so that in $Y\setminus\{y\}$ we have a clopen set which is not the whole space.

\emph{Concluding arguments:}
We now know that in our continuum $Y$ each singleton cell $\{x\}$ is an endpoint and every other point is a cutpoint. By~\cite[Theorem 10.7]{Nadler} we can conclude that $Y$ is in fact a dendrite. Moreover, its endpoints, the singleton cells, form our isometric copy of $X$, that is, $E(Y)=i(X)$.
\end{proof}

\begin{remark}
It is well known that every totally disconnected compact metric space is homeomorphic to a closed subset of the middle thirds Cantor set~\cite[Exercise 7.23]{Nadler}. Moreover, every closed subset of the middle thirds Cantor set is the endpoint set of some subdendrite of the Gehman dendrite~\cite{KKM}. The strength of Theorem~\ref{th:endpoints} is that it gives an isometry rather than a homeomorphism. This is important in Section~\ref{sec:solution} when we address DC3 chaos, which depends on the choice of the metric.
\end{remark}

\section{Extending maps from a closed endpoint set to the whole dendrite}\label{sec:extension}

As our main tool in this section we will use special families of arcs to cover our dendrites, up to their endpoint sets. Since dendrites are known to be uniquely arcwise connected, we will use the notation $[x,y]$ for the unique arc in $X$ with endpoints $x,y$. Arcs in $X$ are called \emph{disjoint out of $x$} if the intersection of any two of them is just $\{x\}$. We will also use the standard notation $A'$ for the set of \emph{accumulation points} of a set $A\subset X$, that is, the points in the closure $\overline{A}$ which are not isolated in $\overline{A}$.

To make our construction work, we will need the following topological characterization from~\cite{Closed}:

\begin{proposition}[\cite{Closed}]\label{prop:char}
A dendrite $X$ has a closed set of endpoints if and only if $X$ does not contain as a subdendrite a homeomorphic copy of either of the following spaces (see Figure~\ref{fig:forbidden}):
\begin{enumerate}
\item The union in the plane of countably many line segments of length tending to zero, pairwise disjoint out of a common point $p$ (called the locally connected fan).
\item The union in the plane of the line segment connecting $(-1,0)$ to $(1,0)$ and the line segments connecting $(\frac1n,\frac1n)$ to $(\frac1n,0)$.
\end{enumerate}
This occurs if and only if
\begin{enumerate}
\item for each $x\in X$ the maximum number of arcs in $X$ disjoint out of $x$ is finite (each point has finite order), and
\item Branchpoints accumulate only at endpoints, that is, $B(X)'\subset E(X)$.
\end{enumerate}
\end{proposition}

The proof of the following lemma is a slight modification of a textbook construction, see~\cite[Theorem 13.27 and Corollary 10.28]{Nadler}. We use the extra condition that $E(X)$ is closed to ensure that the arcs $\alpha_i$ we obtain are free arcs. We give only a sketch of the proof.

\begin{lemma}\label{lem:trees}
Let $X$ be a dendrite containing more than one point and with $E(X)$ closed in $X$. Then there is a sequence of trees $\{p_0\}=T_0\subset T_1 \subset T_2 \subset \cdots$ in $X$ with $\bigcup_i T_i = X\setminus E(X)$ and $\alpha_i=\overline{T_i\setminus T_{i-1}}$ a free arc in $X$ for all $i\in\mathbb{N}$. It follows that
\begin{equation*}
\dist(E(X),T_i)\to 0, \, \dist(\alpha_i, E(X)) \to 0, \, \textnormal{ and } \diam(\alpha_i)\to0 \, \textnormal{ as } i\to\infty.
\end{equation*}
\end{lemma}

\begin{proof}
By Proposition~\ref{prop:char} we can choose a countable set $P\subset X\setminus E(X)$ with $P\supset B(X)$ and $P'=E(X)$. Since $P$ accumulates only at endpoints of $X$, we can choose an enumeration $(p_i)_{i=0}^\infty$ of $P$ such that if $p_j\in[p_0,p_i]$, then $j<i$. Let
\begin{equation*}
T_i=\{p_0\}\cup[p_0,p_1]\cup[p_0,p_2]\cup\cdots\cup[p_0,p_i], \quad i\in\mathbb{N}.
\end{equation*}
Then $\overline{T_i\setminus T_{i-1}}=[p_{j(i)},p_i]=:\alpha_i$ for some $j(i)<i$. From $B(X)\cap (\alpha_i\setminus E(\alpha_i))=\emptyset$ it follows that $\alpha_i$ is a free arc in $X$. From $P'=E(X)$ it follows that $\bigcup_i T_i = X\setminus E(X)$ and therefore $\HausDist(T_i, X)\to 0$ as $i\to\infty$. This immediately gives $\dist(E(X), T_i) \to 0$.

Now let $\epsilon>0$ be given. By the uniform local arcwise connectedness property~\cite[(8.30)]{Nadler} there is $\delta>0$ such that any two points which are $\delta$-close in $X$ are connected by an arc of diameter $<\epsilon$. Find $N$ such that for $i\geq N$, $\HausDist(T_i, X)<\delta$. Let $r_i:X\to T_i$ denote the \emph{first point map} for $T_i$, that is, the retraction sending each point $x$ to the unique point $r_i(x)$ which belongs to every arc $[x,y]$ with $y\in T_i$. For $i\geq N$ and $x\in X\setminus T_i$ there is $y\in T_i$ with $d(x,y)<\delta$, and $[x, r_i(x)] \subset [x,y]$. This shows that $d(r_i(x),x)<\epsilon$ for all $i\geq N$ and all $x\in X$, that is, $r_i$ converges uniformly to the identity on $X$. Since $r_i$ collapses $\alpha_{i+1}$ to a point, this shows that $\diam(\alpha_i)\to 0$ as $i\to\infty$. Together with $P'=E(X)$ this gives $\dist(\alpha_i, E(X)) \to 0$ as well.
\end{proof}

Now we are ready for the main theorem of this section. The proof technique is borrowed from~\cite{Spitalsky}.

\begin{thmb}
Let $X$ be a dendrite with a closed set of endpoints. Any continuous map $f:E(X)\to E(X)$ can be extended to a continuous map $F:X \to X$ in such a way that each point of $X\setminus E(X)$ is eventually fixed.
\end{thmb}

\begin{proof}
Let $T_i$, $\alpha_i$, and $p_i$ be as in the statement and proof of Lemma~\ref{lem:trees}. Choose a sequence of endpoints $e_i\in E(X)$ with $d(p_i,e_i)\to0$. Let $q_0=p_0$ and for $i\geq1$ choose $q_i\in T_{i-1}$ so that $d(q_i,f(e_i))\to0$. Define $F$ on $P$ by $F(p_i)=q_i$. Then extend $F$ to each arc $\alpha_i=[p_i,p_{j(i)}]$ as a homeomorphism onto the arc $[q_i,q_{j(i)}]\subset T_{i-1}$. Finally, on $E(X)$ take $F=f$. Since the arcs $\alpha_i$ are disjoint up to their endpoints and do not intersect $E(X)$, the map $F$ is well-defined.

By Lemma~\ref{lem:trees} we have $d(p_i,p_{j(i)})\to0$. Then $d(e_i,e_{j(i)})\to0$ by construction, $d(f(e_i), f(e_{j(i)}))\to0$ by the continuity of $f$, and $d(q_i,q_{j(i)})\to0$ by construction. Again applying the uniform local arcwise connectedness property~\cite[(8.30)]{Nadler}, we get $\diam(F(\alpha_i))\to0$.

Note that $p_0$ is a fixed point for $F$. For $x\in X\setminus E(X)$ there is $i$ such that $x\in T_i$. Then $F(x)\in T_{i-1}$, $F^2(x)\in T_{i-2}$, and so on until $F^i(x)$ arrives at the fixed point $p_0$.

It remains to show that $F$ is continuous. Since each $\alpha_i$ is a free arc in $X$, continuity is clear on $\alpha_i\setminus E(\alpha_i)$. By Proposition~\ref{prop:char} only finitely many of those free arcs share a given endpoint $p_i\in P$, so we get continuity of $F$ at each $p_i$. Now let $e\in E(X)$. It suffices to consider a sequence $x_n \to e$ with each $x_n\in X\setminus E(X)$. For each $n$ choose an index $i(n)$ so that $x_n\in \alpha_{i(n)}$. Then $i(n)\to\infty$ and $p_{i(n)}\to e$, so we must have $q_{i(n)} \to f(e)$. Since $\diam(F(\alpha_n))\to0$, this implies $F(x_n) \to f(e)=F(e)$ as well.
\end{proof}

Combining Theorems A and B leads to the following corollary.

\begin{corollary}\label{cor:main}
If $(X,d)$ is a totally disconnected compact metric space and $f:X\to X$ is continuous, then there exist a dendrite $(Y,\rho)$, an isometry $i:X\to E(Y)$, and a continuous map $F:Y\to Y$, such that $F\circ i = i\circ f$ and every point $y \in Y\setminus E(Y)$ is eventually fixed under $F$.
\end{corollary}

In other words, every zero-dimensional dynamical system can be isometrically embedded into a system on a dendrite with all extra points eventually fixed.

\section{Application: a dendrite map which is DC3 but not Li-Yorke chaotic}\label{sec:solution}

Let $(X,d)$ be a compact metric space and $f:X\to X$ a continuous map. We refer to the triple $(X,d,f)$ as a \emph{dynamical system}. A pair of points $x,y\in X$ are a \emph{Li-Yorke} pair if the sequence of distances $\big(d(f^i(x),f^i(y))\big)_{i=0}^\infty$, has lower limit zero and upper limit positive. We may also be interested in the statistical distribution of these distances, and so, following~\cite{Stefankova}, we define \emph{lower and upper distribution functions},
\begin{gather*}
\Phi_{x,y}(s)=\liminf \frac1n \# \{ i \leq n ~|~ d(f^i(x),f^i(y)) < s \},\\
\Phi^*_{x,y}(s)=\limsup \frac1n \# \{ i \leq n ~|~ d(f^i(x),f^i(y)) < s \}.
\end{gather*}
The points $x,y$ are a \emph{DC3 pair} if the distribution functions differ on a whole interval, i.e. if there exist $a<b$ such that for all $s\in[a,b]$, $\Phi_{x,y}(s) < \Phi^*_{x,y}(s)$. A set $S\subset X$ is said to be \emph{Li-Yorke scrambled} (resp. \emph{DC3 scrambled}) if all pairs of points $x,y\in S$, $x\neq y$ are Li-Yorke pairs (resp. DC3 pairs). The dynamical system $(X,d,f)$ is called \emph{Li-Yorke chaotic} if it has an uncountable Li-Yorke scrambled set, and it is called \emph{DC3 chaotic} if it has an uncountable DC3 scrambled set. For maps of the interval, trees, and graphs, these two notions are equivalent~\cite{HricMal, Li, Mal07, Smital}. We will show that this equivalence does not hold anymore for dendrite maps.

Let $\Omega=\{0,1\}^{\mathbb{N}_0}$ be the Cantor set and denote $\underline{0}=0000\cdots\in\Omega$. Consider the \emph{adding machine} $\tau(\omega)=\omega+1000\cdots$ with carrying to the right. Equip $\Omega$ with the metric $d_{\Omega}(\omega,\eta)=2^{-\inf\{i~|~\omega_i\neq\eta_i\}}$. Then $\diam(\Omega)=1$ and $\tau:\Omega\to\Omega$ is an isometry.

Choose three decreasing sequences of real numbers $(x_i), (y_i), (z_i)$ indexed by $\mathbb{N}$ with $$x_1=1,\,\, x_i \searrow 0,\,\, y_1=1,\,\, y_i \searrow 0,\,\, z_1=3,\,\, z_i \searrow 2.$$
Consider the subset of $\mathbb{R}^2$ given by
$$P=\{(x_i,y_j)~|~1\leq j\leq i<\infty\}\cup\{(x_i,z_j)~|~1\leq j\leq10^i<\infty\}$$
and let $(p_t)_{t=0}^\infty$ be the enumeration of these points ordered bottom-to-top within each column, then right-to-left column by column, as shown in Figure~\ref{fig:C}. The point at the top of the $n$th column is $p_{T_n}=(x_n,z_1)$, where $T_n=-1+\sum_{i=1}^n (i+10^i)$. Let $Q$ be the reflection of $P$ across the $x$-axis with $q_t$ the reflection of $p_t$. Let $C$ be the closure of $P\cup Q$ in $\mathbb{R}^2$. Let $L=C'$ be the set of accumulation points of $C$. Then $L$ is contained in the $y$-axis and $C$ is the disjoint union $P\cup Q\cup L$. Equip $C$ with the metric $d_C\big((x,y),(x',y')\big)=\max\{|x-x'|,|y-y'|\}$.

\begin{figure}[htb!!]
\scalebox{.7}{
\begin{tikzpicture}[scale=1.5]
\draw [<->, gray, ultra thick] (-0.2,0) -- (6.5,0);
\draw [<->, gray, ultra thick] (0,3.2) -- (0,-3.2);
\foreach \i in {1,...,4} {
 \foreach \j in {1,...,\i} {
  \draw [fill] (6/\i,{pow(10,(\j-1)*(-0.2))}) circle [radius=0.05];
  \draw [fill] (6/\i,{-pow(10,(\j-1)*(-0.2))}) circle [radius=0.05];
 }
 \foreach \j in {1,...,10} {
  \draw [fill] (6/\i,{2+pow(\j,(-0.2))}) circle [radius=0.05];
  \draw [fill] (6/\i,{-2-pow(\j,(-0.2))}) circle [radius=0.05];
 }
 \ifthenelse{\i=1}{}{
  \foreach \j in {11,...,100}{
   \draw [fill] (6/\i,{2+pow(\j,(-0.2))}) circle [radius=0.05];
   \draw [fill] (6/\i,{-2-pow(\j,(-0.2))}) circle [radius=0.05];
  }
 };
 \ifthenelse{\i=1 \OR \i=2}{}{
  \foreach \j in {101,201,...,1000}{
   \draw [fill] (6/\i,{2+pow(\j,(-0.2))}) circle [radius=0.05];
   \draw [fill] (6/\i,{-2-pow(\j,(-0.2))}) circle [radius=0.05];
  }
 };
 \ifthenelse{\i=1 \OR \i=2 \OR \i=3}{}{
  \foreach \j in {1001,2001,...,10000}{
   \draw [fill] (6/\i,{2+pow(\j,(-0.2))}) circle [radius=0.05];
   \draw [fill] (6/\i,{-2-pow(\j,(-0.2))}) circle [radius=0.05];
  }
 };
}
\foreach \j in {1,...,4} {
 \draw [fill] (0,{pow(10,(\j-1)*(-0.2))}) circle [radius=0.05];
 \draw [fill] (0,{-pow(10,(\j-1)*(-0.2))}) circle [radius=0.05];
}
\foreach \j in {1,...,100} {
 \draw [fill] (0,{2+pow(\j,(-0.2))}) circle [radius=0.05];
 \draw [fill] (0,{-2-pow(\j,(-0.2))}) circle [radius=0.05];
}
\foreach \j in {101,201,...,1000} {
 \draw [fill] (0,{2+pow(\j,(-0.2))}) circle [radius=0.05];
 \draw [fill] (0,{-2-pow(\j,(-0.2))}) circle [radius=0.05];
}
\foreach \j in {1001,2001,...,10000} {
 \draw [fill] (0,{2+pow(\j,(-0.2))}) circle [radius=0.05];
 \draw [fill] (0,{-2-pow(\j,(-0.2))}) circle [radius=0.05];
}
\draw [fill] (0,0) circle [radius=0.1];
\draw [fill] (0,2) circle [radius=0.1];
\draw [fill] (0,-2) circle [radius=0.1];
\node [right, shift={(0.1,0)}] at (6,1) {$p_0$};
\node [right, shift={(0.1,0)}] at (6,{2+pow(10,-0.2)}) {$p_1$};
\node [right, shift={(0.1,0)}] at (6,{2+pow(1,-0.2)}) {$p_{10}$};
\node [right, shift={(0.1,0)}] at (3,{pow(10,-0.2)}) {$p_{11}$};
\node [right, shift={(0.1,0)}] at (3,1) {$p_{12}$};
\node [right, shift={(0.1,0)}] at (3,{2+pow(100,-0.2)}) {$p_{13}$};
\node [right, shift={(0.1,0)}] at (3,{2+pow(1,-0.2)}) {$p_{112}$};
\node [right, shift={(0.1,0)}] at (6,-1) {$q_0$};
\node [right, shift={(0.1,0)}] at (6,{-2-pow(10,-0.2)}) {$q_1$};
\node [right, shift={(0.1,0)}] at (6,{-2-pow(1,-0.2)}) {$q_{10}$};
\node [right, shift={(0.1,0)}] at (3,{-pow(10,-0.2)}) {$q_{11}$};
\node [right, shift={(0.1,0)}] at (3,-1) {$q_{12}$};
\node [right, shift={(0.1,0)}] at (3,{-2-pow(100,-0.2)}) {$q_{13}$};
\node [right, shift={(0.1,0)}] at (3,{-2-pow(1,-0.2)}) {$q_{112}$};
\draw [gray, ultra thick] (0,-3.4) -- (0,-3.8) node [below, black] {$0$};
\foreach \i in {1,2,3,4} {
 \draw [gray, ultra thick] (6/\i,-3.6) -- (6/\i,-3.8) node [below, black] {$x_{\i}$};
}
\foreach \j in {-2,0,2} {
 \draw [gray, ultra thick] (-0.4, \j) -- (-0.8, \j) node [left, black] {$\j$};
}
\foreach \j in {1,10,100} {
 \draw [gray, ultra thick] (-0.6, {2+pow(\j,-0.2)}) -- (-0.8, {2+pow(\j,-0.2)}) node [left, black] {$z_{\j}$};
 \draw [gray, ultra thick] (-0.6, {-2-pow(\j,-0.2)}) -- (-0.8, {-2-pow(\j,-0.2)});
}
\foreach \j in {1,2,3} {
 \draw [gray, ultra thick] (-0.6, {pow(10,-0.2*(\j-1))}) -- (-0.8, {pow(10,-0.2*(\j-1))}) node [left, black] {$y_{\j}$};
 \draw [gray, ultra thick] (-0.6, {-pow(10,-0.2*(\j-1))}) -- (-0.8, {-pow(10,-0.2*(\j-1))});
}
\end{tikzpicture}
}
\caption{The fiber space $C$ is countable and compact.}
\label{fig:C}
\end{figure}

Now we construct a skew-product dynamical system. It has the form
\begin{align*}
X&=\Omega \times \Omega \times C &&- \quad \text{``Timer'' $\times$ ``Controller'' $\times$ ``Fiber space''} \\
f(\omega, \eta, c) &= \big(\tau(\omega), \eta, \varphi_{\omega,\eta}(c)\big) &&- \quad \text{``Add. machine'' $\times$ ``Id'' $\times$ ``Fiber map''} \\
\vphantom{\big(\big)}d &= \max\{d_\Omega, d_\Omega, d_C\} && - \quad \text{``Standard product metric''}
\end{align*}
The fiber maps are defined so that $(p_t)$, $(q_t)$ are orbits in the fiber space, except that when we reach one of the top or bottom points $p_{T_n}, q_{T_n}$, we have three options. If the reading on the timer is not close enough to $\tau^{T_n}(\underline{0})$, then we jump directly to the origin. If the timer is okay, then we look at the $n$th symbol in the control sequence $\eta$ to decide which of the sets $P, Q$ to visit next as we start our walk through the $(n+1)$st column. Formally, we have
\begin{gather*}
\varphi_{\omega,\eta}(p_t)=p_{t+1},\,\, \varphi_{\omega,\eta}(q_t)=q_{t+1}\,\,\text{ for }t\not\in\{T_n ~|~ n=1,2,\ldots\}, \\
\varphi_{\omega,\eta}(p_{T_n}) = \varphi_{\omega,\eta}(q_{T_n}) = \begin{cases}
p_{T_n+1}, & \text{ if } d_\Omega(\omega,\tau^{T_n}(\underline{0}))<2^{-n} \text{ and } \eta_n=0, \\
q_{T_n+1}, & \text{ if } d_\Omega(\omega,\tau^{T_n}(\underline{0}))<2^{-n} \text{ and } \eta_n=1, \\
(0,0), & \text{ if } d_\Omega(\omega,\tau^{T_n}(\underline{0}))\geq2^{-n}.
\end{cases}
\end{gather*}
We extend our definition of $\varphi_{\omega,\eta}$ to the limit set $L$ in the unique continuous way. It does not depend on $\omega$ or $\eta$. The points $(0,0)$, $(0,\pm2)$ are fixed, the points $(0,\pm3)$ are mapped to the origin, and each other point is mapped along the $y$-axis and away from the origin to the next point of $L$.

\begin{lemma}
The map $f$ defined above is continuous.
\end{lemma}
\begin{proof}
We need to show that $\varphi$ is continuous as a function of the three variables $\omega, \eta\in\Omega$ and $c\in C$. If $c$ is not one of the top or bottom points $p_{T_n}, q_{T_n},$ or $(0,\pm3)$, then there is a neighborhood of $c$ where $\varphi$ does not depend on $\omega, \eta$, and continuity is clear. If $c$ is one of the points $p_{T_n}, q_{T_n}$, then it is isolated in $C$ and the value of $\varphi_{\omega,\eta}(c)$ only depends on the first $n$ symbols in $\omega, \eta$. If $c=(0,3)$, then the sets $\{p_{T_n} ~|~ n\geq N\}$ are neighborhoods of $c$ in $C$, and the diameter of the set of possible images $\{p_{T_n+1}, q_{T_n+1}, (0,0) ~|~ n\geq N\}$ goes to zero as $N\to\infty$. The same holds for $c=(0,-3)$ using the neighborhoods $\{q_{T_n} ~|~ n\geq N\}$.
\end{proof}

\begin{lemma}\label{lem:pairsonly}
If $f^t(\omega,\eta,c) = (\omega^t,\eta^t,c^t)$, $t\in\mathbb{N}_0$, is an orbit of $f$, then either $c^t$ is eventually constant, or there exists $t_0$ such that $\omega=\tau^{t_0}(\underline{0})$ and $c\in\{p_{t_0},q_{t_0}\}$.
\end{lemma}
\begin{proof}
If $c^t$ ever reaches $L$ then it is eventually fixed. So suppose $c^t\in P\cup Q$ for all $t$. Choose $t_0$ such that $c\in\{p_{t_0}, q_{t_0}\}$. Then $c^t\in\{p_{t_0+t}, q_{t_0+t}\}$ for all $t$. In particular, when $t_0+t=T_n$, we get $d_\Omega(\tau^t(\omega),\tau^{T_n}(\underline{0}))<2^{-n}$. Since $\tau$ is an isometry this reduces to $d_\Omega(\omega,\tau^{t_0}(\underline{0}))<2^{-n}$, and since $n$ is arbitrary we get $\omega=\tau^{t_0}(\underline{0})$.
\end{proof}

\begin{proposition}
The dynamical system $(X,d,f)$ is DC3 but not Li-Yorke chaotic.
\end{proposition}
\begin{proof}
Let $S$ be an uncountable subset of $\Omega$ such that for any two points $\eta, \eta' \in S$ the two sets
\begin{equation*}
N(\eta, \eta') = \{n ~|~ \eta_n = \eta'_n \}, \quad M(\eta, \eta') = \{n ~|~ \eta_n \neq \eta'_n \}
\end{equation*}
are both infinite. We claim that the set $\{\underline{0}\} \times S \times \{p_0\}$ is a DC3 scrambled set for $f$. Let $x=(\underline{0},\eta,p_0)$, $x'=(\underline{0},\eta',p_0)$ be distinct points in this set. Clearly $d_\Omega(\eta,\eta')\leq1$. At time $T_{n-1}$ our points arrive at the top or bottom of column $n-1$, and we look into position $n-1$ in the control sequences $\eta, \eta'$ to see whether the points move to the upper or lower half-plane as they start their walk through column $n$. Therefore we have
\begin{equation*}
n-1\in N(\eta, \eta') \implies \frac{\#\{t\leq T_n ~|~ d(f^t(x), f^t(x'))\leq 1\}}{T_n} \geq \frac{10^n}{T_n}
\end{equation*}
and this converges to $\frac{9}{10}$ as $n$ converges to $\infty$. On the other hand,
\begin{equation*}
n-1\in M(\eta, \eta') \implies \frac{\#\{t\leq T_n ~|~ d(f^t(x), f^t(x'))\leq 4\}}{T_n} \leq \frac{T_n-10^n}{T_n}
\end{equation*}
and this converges to $\frac{1}{10}$ as $n\to\infty$. Since both situations occur for infinitely many $n$, it follows that $\Phi_{x,x'}(s)\leq\frac{1}{10}\leq\frac{9}{10}\leq\Phi^*_{x,x'}(s)$ on the whole interval $s\in(1,4)$. Thus $x,x'$ are a DC3 pair. This completes the proof that $(X,d,f)$ is DC3 chaotic.

Now suppose that two points $x=(\omega,\eta,c)$, $x'=(\omega',\eta',c')$ in $X$ are a Li-Yorke pair. Since the adding machine and the identity map are both isometries on $\Omega$, the proximality condition $\liminf d(f^t(x),f^t(x')) = 0$ implies that $\omega=\omega'$ and $\eta=\eta'$. Since the set $C$ is only countably infinite, this shows that all Li-Yorke scrambled sets for $f$ are at most countable. In other words, $(X,d,f)$ is not Li-Yorke chaotic.
\end{proof}

\begin{remark}
There are Li-Yorke pairs for $f$, such as the pair $(\underline{0},\underline{0},p_0)$, $(\underline{0},\underline{0},(0,0))$. But using Lemma~\ref{lem:pairsonly} it is easy to show that any Li-Yorke scrambled set for $f$ has cardinality at most 2.
\end{remark}

\begin{thmc}
There is a dynamical system on a dendrite which is DC3 but not Li-Yorke chaotic.
\end{thmc}
\begin{proof}
The space $X$ constructed above is a product of totally disconnected compact metric spaces, and is therefore itself a totally disconnected compact metric space. Thus we can apply Corollary~\ref{cor:main} to construct a dendrite map $(Y,\rho,F)$ containing an isometric copy of the system $(X,d,f)$, and with all other points eventually fixed. The DC3 scrambled set for $f$ is also present for $F$, which gives us DC3 chaos. On the other hand, two eventually fixed points cannot form a Li-Yorke pair, so any Li-Yorke scrambled set for $F$ is made of a Li-Yorke scrambled set for $f$ with at most one new point adjoined. Since $f$ has no uncountable (nor even infinite) Li-Yorke scrambled sets, neither does $F$.
\end{proof}

\end{document}